\documentclass{amsart}

\usepackage[dvipsnames]{xcolor}
\usepackage{amsmath, amsthm, amssymb}
\usepackage{mathtools}
\usepackage{enumerate} 
\usepackage{hyperref}

\hypersetup{colorlinks,
	linkcolor=blue!70!black,
	citecolor=green!50!black,
	urlcolor=red!50!black
}

\theoremstyle{plain}
    \newtheorem{lemma}{Lemma}

    \newtheorem{theorem}{Theorem}

\DeclarePairedDelimiter\ceil{\lceil}{\rceil}

\DeclarePairedDelimiter\abs{\lvert}{\rvert}

\newcommand{\Z}{\mathbb{Z}}
\newcommand{\N}{\mathbb{N}}
\newcommand{\R}{\mathbb{R}}

\renewcommand{\phi}{\varphi}

\title[Lebesgue additive square problem]{A Lebesgue variant of the additive square problem}



\author[I. Vukusic]{Ingrid Vukusic}
\address{I. Vukusic,
University of Waterloo,
200 University Ave W,
Waterloo, ON N2L 3G1, 
Canada}
\email{ingrid.vukusic\char'100uwaterloo.ca}

\begin{document}

\begin{abstract}
We prove that for every Lebesgue measurable function $f \colon \R \to \{v_1, \ldots, v_t\} \subseteq \R$ there exist $x \in \R$, $y>0$ such that 
\[
    \int_x^{x+y} f\, d \lambda   
    = \int_{x+y}^{x+2y} f\, d \lambda.
\]
This can be seen as a variant of the classical additive square problem for infinite words over finite alphabets. In other words, we prove that additive squares cannot be avoided in the Lebesgue setting.  
\end{abstract}

\maketitle

\section{Introduction}

The additive square problem is a relatively famous open problem in the area of combinatorics on words:
Does there exist an infinite word 
over a finite alphabet $A\subseteq \N$, such that no two consecutive blocks of the same length have the same sum? 
In other words, can we avoid a factor (i.e., a contiguous block of symbols within the word) of the shape
\[
    a_1 a_2 \cdots a_k b_1 b_2 \cdots b_k,
    \quad \text{with }
    \sum_{i = 1}^k a_i = \sum_{i = 1}^k b_i.
\]
A factor of the above shape is called an \textit{additive square} because it can be seen as a generalization of a \textit{square}, which is simply a factor of the shape $a_1 a_2 \cdots a_ka_1 a_2 \cdots a_k$ $= ww$.
In 1906, Thue \cite{Thue1906} constructed an infinite word over a 3-letter alphabet that avoids squares.
Since then, similar and other avoidability problems have received much attention.
For example, Erd\H{o}s \cite{Erdos1961} asked about the avoidability of \textit{abelian squares}, that is, factors of the shape $w \widetilde{w}$, where $\widetilde{w}$ is a permutation of $w$. (This problem has been completely solved \cite{Keranen1992}.)

The additive square problem was raised independently by several authors \cite{BrownFreedman1987, PirilloVarricchio1994, HalbeisenHungerbuhler2000}, and variations and related questions have been considered, for example, in \cite{ArdalBrownJungicVeselin2012, AuRobertsonShallit2011, Brown2012, DuMousaviRowlandSchaefferShallit2017}.
In particular, Rao \cite{Rao2015} used the construction by Cassaigne et al.\ \cite{CassaigneCurrieSchaefferShallit2014} to show that additive cubes (i.e., factors of the shape $ww'w''$, where each block has the same length and sum) can be avoided over certain alphabets of size 3. However, it is still unknown whether additive squares can be avoided.

In this paper, we consider a Lebesgue version of the additive square problem. Note that an infinite word over a finite alphabet can be interpreted as a function $g \colon \N \to \{a_1, \ldots, a_t\} \subseteq \N$, where $\N$ denotes the positive integers. In this sense, the function $g$ avoids additives squares if 
\[
    \sum_{i = m}^{m+k-1} g(i)
    \neq \sum_{i = m+k}^{m+2k-1} g(i)
    \quad \text{for all } m \in \N, k \in \N.
\]
Therefore, the natural Lebesgue variant of the additive square problem is as follows: 
Does there exist a Lebesgue measurable function $f\colon \R \to \{v_1, \ldots, v_t\} \subseteq \R$ such that
\begin{equation}\label{eq:avoid_square}
    \int_x^{x+y} f\, d \lambda  
    \neq \int_{x+y}^{x+2y} f\, d \lambda
    \quad \text{for all } x \in \R \text{ and } y >0?
\end{equation}
Here, we can interpret $f$ as an ``uncountable word'', defined over the finite alphabet $\{v_1, \ldots, v_t\}$ and indexed by $\R$. Other transfinite words have been studied in 
\cite{BoassonCarton2020, CartonChoffrut2001}, where the words are indexed over the countable ordinals (which also form an uncountable set).
In this note, we show that in the Lebesgue setting additive squares cannot be avoided; i.e., 
the answer to the question involving \eqref{eq:avoid_square} is ``no''.

\medskip

Before we present the proof, let us give a brief intuitive explanation.
Assume that we have a function $f\colon \R \to \{v_1, \ldots, v_t\} \subseteq \R$ which avoids additive squares in the sense of \eqref{eq:avoid_square}.
Then clearly $f$ cannot be constant on any interval. This means that the preimages $A_1, \ldots , A_t$ of $v_1, \ldots, v_t$ cannot contain (up to Lebesgue zero sets) any intervals. Moreover, morally speaking, the ``densities'' of the preimages have to be ``constantly changing''. But this 
contradicts the properties of the Lebesgue measure: Lebesgue's density theorem says that every Lebesgue measurable set has density 1 almost everywhere.

\section{Result and Proof}

\begin{theorem}\label{thm:lebSquares}
Let $f\colon \R \to \R$ be a Lebesgue measurable function that
takes only finitely many values in $\R$. Then there exist $x \in \R$, $y>0$ such that 
\[
    \int_x^{x+y} f\, d \lambda   
    = \int_{x+y}^{x+2y} f\, d \lambda.
\]
\end{theorem}

We start by arguing that if all values of ``consecutive integrals''
were distinct, they would either have to be always strictly increasing from left to right, or strictly decreasing.

\begin{lemma}\label{lem:alwaysSmallerOrLarger}
Let $f \colon \R \to \R$ be a Lebesgue measurable and bounded function. Assume that 
\[
    \int_x^{x+y} f\, d \lambda \neq \int_{x+y}^{x+2y} f d \lambda
    \quad \text{for all } x\in \R \text{ and } y>0.
\]
Then either
\[
    \int_x^{x+y} f\, d \lambda 
    < \int_{x+y}^{x+2y} f\, d \lambda 
    \quad \text{ for all } x \in \R \text{ and } y>0
\]
or 
\[
    \int_x^{x+y} f\, d \lambda 
    > \int_{x+y}^{x+2y} f\, d \lambda 
    \quad \text{ for all } x \in \R \text{ and } y>0.
\]
\end{lemma}
\begin{proof}
This follows from a simple continuity argument. 
To be precise, 
assume the opposite, i.e., 
\begin{equation}\label{eq:assIntsNeq}
    \int_x^{x+y} f\, d \lambda 
    \neq \int_{x+y}^{x+2y} f\, d \lambda 
    \quad \text{ for all } x \in \R \text{ and } y>0
\end{equation}  
and that there exist $a,c \in \R$, $b,d>0$ such that
\begin{equation}\label{eq:assInts}
    \int_a^{a+b} f\, d \lambda 
    < \int_{a+b}^{a+2b} f d \lambda 
    \quad \text{and} \quad 
    \int_c^{c+d} f\, d \lambda 
    > \int_{c+d}^{c+2d} f d \lambda.
\end{equation}
Consider the function
\[
    F(t) 
    := \int_{a+t(c-a)}^{a+b+t(c+d-a-b)} f\, d \lambda - \int_{a+b+t(c+d-a-b)}^{a+2b+t(c+2d-a-2b)} f\, d \lambda,
    \quad \text{for } t \in [0,1].
\]
Note that by \eqref{eq:assInts}, we have $F(0)<0$ and $F(1)>0$. 
Moreover, since $f$ is bounded, $F$ is continuous.
Therefore, there exists $t_0 \in (0,1)$ such that $F(t_0) = 0$. But this means
\[
    \int_{a+t_0(c-a)}^{a+t_0(c-a) + b+t_0(d-b)} f\, d \lambda 
    = \int_{a+t_0(c-a) + b+t_0(d-b)}^{a+t_0(c-a) + 2(b+t_0(d-b))} f\, d \lambda,
\]
which contradicts the assumption \eqref{eq:assIntsNeq}, as $a+t_0(c-a) \in \R$ and $b+t_0(d-b) = (1-t_0)b + t_0 d >0$.
\end{proof}

Next, we recall Lebesgue's density theorem; for an elementary proof see, e.g., \cite[Appendix D]{RooijSchikhof1982}.

\begin{lemma}[Lebesgue's density theorem]\label{lem:lebDensity} 
Let $A \subset \R$. Then for almost all $a \in A$ (in a Lebesgue measure sense) we have
\[
	\lim_{r \downarrow 0} \frac{\lambda(A \cap (a-r, a+r))}{2 r} = 1.
\]
\end{lemma}

Now we are ready to prove that in the Lebesgue setting additive squares cannot be avoided.

\begin{proof}[Proof of Theorem~\ref{thm:lebSquares}]

The proof is by induction on the number of values of $f$.
If $f \colon \R \to \{ v_1\}$, there is nothing to do.
Now assume that the theorem is proven for all $f \colon \R \to \{v_1, \ldots, v_{t-1}\}$ for some $t\geq 2$.
Let $f \colon \R \to \{v_1, \ldots , v_t\}$.
After a linear transformation, we may assume that $0 = v_1 < v_2 < \dots < v_t = 1$.
Let $A_1, \ldots, A_t$ be the preimages of $v_1, \ldots, v_t$.
We may assume $\lambda(A_i) > 0$ for all $i$ because otherwise we can use the induction hypothesis.
We want to show that there exist $x \in \R$, $y>0$ such that 
\begin{equation}\label{eq:wts}
    \int_x^{x+y} f\, d \lambda   
    = \int_{x+y}^{x+2y} f\, d \lambda.
\end{equation}
Suppose \eqref{eq:wts} fails to hold for all $x,y$.
Then by Lemma~\ref{lem:alwaysSmallerOrLarger} we may assume w.l.o.g.
\begin{equation}\label{eq:integrals_LeftLargerThanRight}
    \int_x^{x+y} f\, d \lambda 
    > \int_{x+y}^{x+2y} f\, d \lambda 
    \quad \text{ for all } x \in \R \text{ and } y>0.
\end{equation}
Since $\lambda(A_1)>0$, there must exist $n \in \Z$ with 
$\lambda(A_1 \cap [n,n+1]) > 0$. 
Assume w.l.o.g.\ that $n = 0$, i.e., $\lambda(A_1 \cap [0,1]) > 0$.
By \eqref{eq:integrals_LeftLargerThanRight} we have 
$
    \int_1^2 f\, d \lambda
    > \int_2^3 f\, d \lambda
    \geq 0,
$
so we can set 
\[
    \int_1^2 f\, d \lambda =:c >0.
\]
Since $\lambda(A_1 \cap [0,1]) > 0$, Lemma~\ref{lem:lebDensity} implies that there exist
$x \in (0,1)$ and $r >0$ such that
\[
    \lambda(A_1 \cap (x-r, x+r)) 
    \geq 2r ( 1 - \frac{c}{6}).
\]
Now let $I = [a,b] := [x-r, x+r]$.
Then the above inequality becomes $\lambda(A_1 \cap I) \geq \abs{I} ( 1 - \frac{c}{6})$ and, since $A_1$ is the preimage of $0$, we obtain
\[
    \int_{I} f\, d \lambda
    \leq 0 \cdot \lambda(A_1 \cap \abs{I}) + 1 \cdot (\abs{I} - \lambda(A_1 \cap \abs{I})) 
    \leq \frac{c}{6} \cdot \abs{I}.
\]
Let $d : = \abs{I}$ and we consider the shifts of the interval $I$ given by $I_k := [a+kd, b+ kd]$ for $k = 0, 1, 2, \ldots$.
By \eqref{eq:integrals_LeftLargerThanRight}, we have
\[
    \int_{I} f\, d \lambda
    = \int_{I_0} f\, d \lambda
    > \int_{I_1} f\, d \lambda
    > \int_{I_2} f\, d \lambda
    > \dots.
\]
Let $N$ be the number of intervals $I_k$ that intersect $[1,2]$.
Then
\[
    N \leq \ceil{1/\abs{I}} + 1
    \leq \frac{3}{\abs{I}}.
\]
Assume that $I_{k_0}, I_{k_0 + 1}, \ldots , I_{k_0 + N - 1}$ are the intervals that intersect $[1,2]$. Then we have
\begin{align*}
    c
    &= \int_1^2 f\, d \lambda
    \leq \int_{I_{k_0}} f\, d \lambda + \dots + \int_{I_{k_0 + N - 1}} f\, d \lambda \\
    &\leq N \cdot \int_{I} f\, d \lambda
    \leq \frac{3}{\abs{I}} \cdot \frac{c\cdot \abs{I}}{6} 
    = \frac{c}{2};
\end{align*}
a contradiction.

\end{proof}

\section*{Acknowledgments}

The author would like to thank Jeffrey Shallit for suggesting the problem, for his encouragement, 
and for several much-appreciated comments on the nuances of English grammar.
This work was supported by the NSERC under Grant 2024-03725.

\bibliographystyle{habbrv}
\bibliography{refs}

\begin{thebibliography}{10}
\expandafter\ifx\csname url\endcsname\relax
  \def\url#1{\texttt{#1}}\fi
\expandafter\ifx\csname doi\endcsname\relax
  \def\doi#1{\burlalt{doi:#1}{http://dx.doi.org/#1}}\fi
\expandafter\ifx\csname urlprefix\endcsname\relax\def\urlprefix{URL: }\fi
\expandafter\ifx\csname href\endcsname\relax
  \def\href#1#2{#2}\fi
\expandafter\ifx\csname burlalt\endcsname\relax
  \def\burlalt#1#2{\href{#2}{#1}}\fi

\bibitem{ArdalBrownJungicVeselin2012}
H.~Ardal, T.~Brown, V.~Jungi\'c, and J.~Sahasrabudhe.
\newblock On abelian and additive complexity in infinite words.
\newblock {\em Integers}, 12(5):795--804, 2012.
\newblock \doi{10.1515/integers-2012-0005}.

\bibitem{AuRobertsonShallit2011}
Y.-H. Au, A.~Robertson, and J.~Shallit.
\newblock van der {W}aerden's theorem and avoidability in words.
\newblock {\em Integers}, 11:A7, 15, 2011.
\newblock \doi{10.1515/INTEG.2011.007}.

\bibitem{BoassonCarton2020}
L.~Boasson and O.~Carton.
\newblock Transfinite {L}yndon words.
\newblock {\em Log. Methods Comput. Sci.}, 16(4):Paper No. 9, 38, 2020.

\bibitem{Brown2012}
T.~Brown.
\newblock Approximations of additive squares in infinite words.
\newblock {\em Integers}, 12(5):805--809, 2012.
\newblock \doi{10.1515/integers-2012-0006}.

\bibitem{BrownFreedman1987}
T.~C. Brown and A.~R. Freedman.
\newblock Arithmetic progressions in lacunary sets.
\newblock {\em Rocky Mountain J. Math.}, 17(3):587--596, 1987.
\newblock \doi{10.1216/RMJ-1987-17-3-587}.

\bibitem{CartonChoffrut2001}
O.~Carton and C.~Choffrut.
\newblock Periodicity and roots of transfinite strings.
\newblock {\em Theor. Inform. Appl.}, 35(6):525--533, 2001.
\newblock \doi{10.1051/ita:2001102}.
\newblock A tribute to Aldo de Luca.

\bibitem{CassaigneCurrieSchaefferShallit2014}
J.~Cassaigne, J.~D. Currie, L.~Schaeffer, and J.~Shallit.
\newblock Avoiding three consecutive blocks of the same size and same sum.
\newblock {\em J. ACM}, 61(2):Art. 10, 17, 2014.
\newblock \doi{10.1145/2590775}.

\bibitem{DuMousaviRowlandSchaefferShallit2017}
C.~F. Du, H.~Mousavi, E.~Rowland, L.~Schaeffer, and J.~Shallit.
\newblock Decision algorithms for {F}ibonacci-automatic words, {II}: {R}elated
  sequences and avoidability.
\newblock {\em Theoret. Comput. Sci.}, 657:146--162, 2017.
\newblock \doi{10.1016/j.tcs.2016.10.005}.

\bibitem{Erdos1961}
P.~Erd\H{o}s.
\newblock Some unsolved problems.
\newblock {\em Magyar Tud. Akad. Mat. Kutat\'o{} Int. K\"ozl.}, 6:221--254,
  1961.

\bibitem{HalbeisenHungerbuhler2000}
L.~Halbeisen and N.~Hungerb\"uhler.
\newblock An application of van der {W}aerden's theorem in additive number
  theory.
\newblock {\em Integers}, 0:A7, 5, 2000.

\bibitem{Keranen1992}
V.~Ker\"anen.
\newblock Abelian squares are avoidable on {$4$} letters.
\newblock In {\em Automata, languages and programming ({V}ienna, 1992)}, volume
  623 of {\em Lecture Notes in Comput. Sci.}, pages 41--52. Springer, Berlin,
  1992.
\newblock \doi{10.1007/3-540-55719-9\_62}.

\bibitem{PirilloVarricchio1994}
G.~Pirillo and S.~Varricchio.
\newblock On uniformly repetitive semigroups.
\newblock {\em Semigroup Forum}, 49(1):125--129, 1994.
\newblock \doi{10.1007/BF02573477}.

\bibitem{Rao2015}
M.~Rao.
\newblock On some generalizations of abelian power avoidability.
\newblock {\em Theoret. Comput. Sci.}, 601:39--46, 2015.
\newblock \doi{10.1016/j.tcs.2015.07.026}.

\bibitem{Thue1906}
A.~Thue.
\newblock {\"U}ber unendliche {Zeichenreihen}.
\newblock Christiana {Vidensk}. {Selsk}. {Skr}. 1906. {Nr}. 7. 22 {S}. {Lex}.
  {{\(8^{\circ}\)}} (1906)., 1906.
\newblock Reprinted in Selected Mathematical Papers of Axel Thue, edited by T.
  Nagell, Universitetsforlaget, Oslo (1977) 413-478.

\bibitem{RooijSchikhof1982}
A.~C.~M. van Rooij and W.~H. Schikhof.
\newblock {\em A second course on real functions}.
\newblock Cambridge University Press, Cambridge-New York, 1982.

\end{thebibliography}

\end{document}